\documentclass[11pt]{amsart}
\usepackage{graphicx,amssymb,amsmath,amsthm}
\usepackage{enumerate}
\usepackage{dsfont}
\usepackage[colorlinks, citecolor=red]{hyperref}
\usepackage{mathrsfs}
\usepackage{epsfig}
\usepackage{verbatim}
\usepackage{lscape}
\usepackage{subfigure}
\usepackage{epstopdf}
\usepackage{url}
\usepackage{caption}
\usepackage{algpseudocode}
\usepackage{geometry}
\usepackage{amsfonts}
\usepackage{graphicx,cite}
\usepackage{paralist}
\usepackage{url}

\newcommand{\innerp}[1]{\langle {#1} \rangle}

\newcommand{\abs}[1]{\lvert#1\rvert}

\newcommand\Tr{\text{Tr}}

\newcommand{\E}{{\mathbb E}}
\newcommand{\ve}{{\mathbf e}}
\newcommand{\vz}{{\mathbf z}}
\newcommand{\vt}{{\mathbf t}}

\newcommand{\vu}{{\mathbf u}}

\newcommand{\vv}{{\mathbf v}}
\newcommand{\vx}{{\mathbf x}}
\newcommand{\Disc}{{\rm Disc}}
\newcommand{\vw}{{\mathbf w}}
\newcommand{\vh}{{\mathbf h}}

\newcommand{\R}{{\mathbb R}}

\newcommand{\cS}{{\mathcal S}}
\newcommand{\cD}{{\mathcal D}}
\newcommand{\MC}{{\mathbb C}}
\newcommand{\cC}{{\mathcal C}_0}

\numberwithin{equation}{section}

\theoremstyle{definition}
\newtheorem{definition}{Definition}[section]

\theoremstyle{plain}
\newtheorem{corollary}[definition]{Corollary}
\newtheorem{prop}[definition]{Proposition}
\newtheorem{theorem}[definition]{Theorem}
\newtheorem{lemma}[definition]{Lemma}

\newtheorem{conjecture}[definition]{Conjecture}
\theoremstyle{remark}
\newtheorem{remark}[definition]{Remark}

\date{}

\begin{document}
\baselineskip 18pt
\title[Upper and Lower bounds  for matrix discrepancy]
{Upper and Lower bounds  for matrix discrepancy }
\author{Jiaxin Xie}
\address{LMIB of the Ministry of Education, School of Mathematical Sciences, Beihang University, Beijing, 100191, China}
\email{xiejx@buaa.edu.cn}
\author{Zhiqiang Xu}
\address{LSEC, Inst.~Comp.~Math., Academy of
Mathematics and System Science,  Chinese Academy of Sciences, Beijing, 100091, China \newline
School of Mathematical Sciences, University of Chinese Academy of Sciences, Beijing 100049, China}
\email{xuzq@lsec.cc.ac.cn}
\author{Ziheng Zhu}
\address{LSEC, Inst.~Comp.~Math., Academy of
Mathematics and System Science,  Chinese Academy of Sciences, Beijing, 100091, China \newline
School of Mathematical Sciences, University of Chinese Academy of Sciences, Beijing 100049, China}
\email{zhuziheng@lsec.cc.ac.cn}

\begin{abstract}
The aim of this paper is to study the matrix discrepancy problem.
Assume that $\xi_1,\ldots,\xi_n$ are independent scalar random variables with finite support and $\vu_1,\ldots,\vu_n\in
\MC^d$.
Let $\cC$ be the minimal constant for which the following holds:
\[
 \Disc(\vu_1\vu_1^*,\ldots,\vu_n\vu_n^*; \xi_1,\ldots,\xi_n)\,\,:=\,\,\min_{\varepsilon_1\in \cS_1,\ldots,\varepsilon_n\in \cS_n}\bigg\|\sum_{i=1}^n\mathbb{E}[\xi_i]\vu_i\vu_i^*-\sum_{i=1}^n\varepsilon_i\vu_i\vu_i^*\bigg\|\leq \cC\cdot\sigma,
\]
where $\sigma^2 = \big\|\sum_{i=1}^n \mbox{{\bf Var}}[\xi_i](\vu_i\vu_i^*)^2\big\|$ and $\cS_j$ denotes the support of $\xi_j, j=1,\ldots,n$.
Motivated by the technology developed by Bownik, Casazza, Marcus, and Speegle
\cite{bownik2019improved}, we prove  $\cC\leq 3$. This improves  Kyng, Luh and Song's method with
which $\cC\leq 4$ \cite{kyng2019four}.
For the case where $\{\vu_i\}_{i=1}^n\subset \MC^d$ is a unit-norm tight frame with $ n\leq 2d-1$ and $\xi_1,\ldots,\xi_n$ are independent Rademacher random variables, we  present the exact value of $\Disc(\vu_1\vu_1^*,\ldots,\vu_n\vu_n^*; \xi_1,\ldots,\xi_n)=\sqrt{\frac{n}{d}}\cdot\sigma$, which implies $\cC\geq \sqrt{2}$.
\end{abstract}

\maketitle

\section{Introduction}

\subsection{Problem setup}
\emph{Matrix discrepancy} is an active topic recently, which has numerous applications in
mathematics and computer science (see
\cite{chazelle2001discrepancy,matousek2009geometric,CST14}). It also has  deep
connections with many topics  in mathematics which includes the  Kadison-Singer problem
\cite{Marcus2015InterlacingII,kyng2019four}, the Lyapunov-type theorem
\cite{akemann2014a}  and Spencer's ``six standard deviations" theorem
\cite{SpencerSix} etc.

 Assume that $\xi_1,\ldots,\xi_n$ are independent scalar random variables with finite
 support. For $j=1,\ldots,n$, we use $\cS_j$ to denote the support of $\xi_j$. Let $\mathbf{A}_1,\ldots, \mathbf{A}_n\in \MC^{d\times
 d}$ be Hermitian matrices.
 The aim of matrix discrepancy  is to estimate
 $$
 \Disc(\mathbf{A}_1,\ldots,\mathbf{A}_n; \xi_1,\ldots,\xi_n)\,\,:=\,\,\min_{\varepsilon_1\in \cS_1,\ldots,\varepsilon_n\in \cS_n}\bigg\| \sum_{j=1}^n \varepsilon_j\mathbf{A}_j-\sum_{j=1}^n\E[\xi_j]\mathbf{A}_j \bigg\|,
 $$
 where $\|\cdot\|$ denotes the standard spectral matrix norm.
 We call $\Disc(\mathbf{A}_1,\ldots,\mathbf{A}_n; \xi_1,\ldots,\xi_n)$ the {\em matrix discrepancy}
corresponding to matrices $\mathbf{A}_1,\ldots,\mathbf{A}_n$ and random variables $\xi_1,\ldots,\xi_n$.
In \cite{Oli10} (see also \cite[Theorem 1.2]{Tro12}),  Oliveira employed a probability method to prove that
\begin{equation}\label{eq:tropp}
	\Disc(\mathbf{A}_1,\ldots,\mathbf{A}_n; \xi_1,\ldots,\xi_n) \leq O(\sqrt{\log d})\sigma,
\end{equation}
where  $\xi_1,\ldots,\xi_n$ are independent Rademacher random variables, i.e., uniformly distributed in $\{\pm 1\}$ and $\sigma^2 = \big\|\sum_{i=1}^n \mathbf{A}_i^2\big\|$.


\subsection{Our contribution}

\subsubsection{An improved upper bound for rank-1 matrices}\label{subsection-s-norm}
Under some conditions on the matrices $\{\mathbf{A}_i\}_{i=1}^n$,
it is possible to remove the factor $ \sqrt{\log d}$  on the right side of (\ref{eq:tropp}).
Recently, Kyng,
Luh and Song \cite{kyng2019four} considered the case of Hermitian matrices $\mathbf{A}_1,\ldots,\mathbf{A}_n$ with ${\rm rank}(\mathbf{A}_i)=1$ for $1\leq i\leq n$.
Let $\cC$ be the minimal constant  satisfying
\begin{equation}\label{eq:c0}
\Disc(\vu_1\vu_1^*,\ldots,\vu_n\vu_n^*; \xi_1,\ldots,\xi_n)\leq \cC\cdot\sigma,
\end{equation}
for all the vectors $\vu_1,\ldots,\vu_n\in\mathbb{C}^d$ and independent random variables $\xi_1,\ldots,\xi_n$ with finite support, where $\sigma^2 = \big\|\sum_{i=1}^n \mbox{{\bf Var}}[\xi_i](\vu_i\vu_i^*)^2\big\|$.
In \cite{kyng2019four}, Kyng, Luh and Song showed that one can take $\cC\leq 4$.

\begin{theorem}[{\cite[Theorem 1.4]{kyng2019four}}]\label{th:old}
	Suppose that $\xi_1,\ldots,\xi_n$ are independent scalar random variables with finite
	support.  Let $\vu_1,\ldots,\vu_n\in\mathbb{C}^d$ and
	$$
	\sigma^2 = \bigg\|\sum_{i=1}^n \mbox{{\bf Var}}[\xi_i](\vu_i\vu_i^*)^2\bigg\|.
	$$
	Then
	\begin{equation} \label{eq:old}
		\Disc(\vu_1\vu_1^*,\ldots,\vu_n\vu_n^*;\,\, \xi_1,\ldots,\xi_n)\,\, \leq\,\, 4\cdot\sigma.
	\end{equation}
\end{theorem}

As mentioned in \cite{kyng2019four}, this theorem  strengthens the Kadison-Singer theorem due to Marcus, Spielman and Srivastava \cite{Marcus2015InterlacingII}.  Kyng, Luh and
Song also used Theorem \ref{th:old} to
improve the original Lyapunov-type theorem in \cite{akemann2014a}, which says that the image of $[0,1]^n$ under the map $$\Phi:(t_1,\ldots,t_n)\mapsto\sum\limits_{i=1}^nt_i\vu_i\vu_i^*$$
can be
approximated by the image of its vertex set $\{0,1\}^n$.
\begin{corollary}[{\cite[Corollary 1.7]{kyng2019four}}]\label{co:old}
	Let $\vu_1,\ldots,\vu_n\in\mathbb{C}^d$ which satisfy
	$\|\sum_{i=1}^n \vu_i\vu_i^*\|\leq 1$ and $\|\vu_i\|^2\leq\epsilon$ for all $i$.
	Then for any $t_i\in [0,1], 1\leq i \leq n$, there exists  a  subset
	$S\subset \{ 1,2,\ldots,n\}$ such that
	\begin{equation*}
		\bigg\| \sum_{i\in S}\vu_i\vu_i^*-\sum_{i=1}^n t_i\vu_i\vu_i^*\bigg\|\,\,\leq\,\, 2\sqrt{\epsilon}.
	\end{equation*}
\end{corollary}
One of  aims of this paper is to improve the bound $4\cdot\sigma$ in (\ref{eq:old}) to
$3\cdot\sigma$ showing that one can take $\cC\leq 3$. To do that, we have to employ the
technology developed in \cite{bownik2019improved}. Particularly, we have
\begin{theorem}\label{th:main}
	Under the assumptions of  Theorem \ref{th:old}, we have
	\begin{equation} \label{eq:new}
		\Disc(\vu_1\vu_1^*,\ldots,\vu_n\vu_n^*;\,\, \xi_1,\ldots,\xi_n)\,\, \leq\,\, 3\cdot\sigma.
	\end{equation}
\end{theorem}

Similarly with  \cite{kyng2019four}, we take $\xi_1,\ldots,\xi_n$ as independent
$\{0,1\}$-valued  random variables with $\mathbb{E}[\xi_i] =t_i\in[0,1]$. Then
$\mbox{{\bf Var}}[\xi_i]=t_i(1-t_i)\leq \frac{1}{4}$. If
$\max_i\|\vu_i\|^2\leq\epsilon$ and $\sum\limits_{i=1}^n\vu_i\vu_i^*=\mathbf{I}$, then we have
$\sigma^2\leq \epsilon/4$. Thus by Theorem \ref{th:main}, we have the following
improved Lyapunov-type  theorem \cite{akemann2014a}. This, in turn, improves
Corollary \ref{co:old}.
\begin{corollary}
	Let $\vu_1,\ldots,\vu_n\in\mathbb{C}^d$ which satisfy
	$\|\sum_{i=1}^n \vu_i\vu_i^*\|\leq 1$ and $\|\vu_i\|^2\leq\epsilon$ for all $i$.
	Then for any $t_i\in [0,1], 1\leq i \leq n$, there exists  a  subset
	$S\subset \{ 1,2,\ldots,n\}$ such that
	\begin{equation*}
		\bigg\| \sum_{i\in S}\vu_i\vu_i^*-\sum_{i=1}^n t_i\vu_i\vu_i^*\bigg\|\,\,\leq\,\, \frac{3}{2}\sqrt{\epsilon}.
	\end{equation*}
\end{corollary}

\subsubsection{Matrix discrepancy for tight frame.}
In this subsection, we consider the case where $\{\vu_i\}_{i=1}^n$ form a tight frame in ${\mathbb C}^d$.
 We say that 	vectors  $\{\vu_1,\ldots,\vu_n\}\subset\mathbb{C}^d$ form a  \emph{tight frame} in
$\mathbb{C}^d$ if there exists a constant $C>0$ such that
\[
\sum\limits_{i=1}^n\vu_i\vu_i^* = C \cdot \mathbf{I},
\] where $C$ is called the \emph{frame bound}.  For more details about tight frame, we refer the reader to \cite{Wal18}.
Futhermore, if unit vectors $\{\vu_1,\ldots,\vu_n\}\subset \mathbb{C}^d$ form a tight frame with frame bound $\frac{n}{d}$, i.e.
\begin{equation}\label{eq:FNTF}
\sum\limits_{i=1}^n\vu_i\vu_i^*=\frac{n}{d}\cdot\mathbf{I}\quad \text{and} \quad\|\vu_i\|=1\,\, \text{for}\,\, 1\leq i\leq n,
\end{equation}
then we call  that the  vectors $\{\vu_1,\ldots,\vu_n\}$ form a \emph{unit-norm tight frame} in $\mathbb{C}^d$.
For the exsistence and characterization of unit-norm tight frames, we refer the reader to \cite{GKK01,zim01,bf03}.
We next present an upper bound for $\Disc(\mathbf{u}_1\mathbf{u}_1^*,\ldots,\mathbf{u}_n\mathbf{u}_n^*;\xi_1,\ldots,\xi_n)$
 under the assumption of $\{\vu_1,\ldots,\vu_n\}$  being a tight frame and an exact value for $\{\vu_1,\ldots,\vu_n\}$  being a unit-norm tight frame.
\begin{theorem}\label{thm:sharp}
	Suppose that $d$ and $n$ are two positive integers  and $\xi_1,\ldots,\xi_n$ are independent Rademacher  random variables.
	Suppose that  $\{\vu_1,\ldots,\vu_n\} \subset \mathbb{C}^d$  is a tight frame in ${\mathbb C}^d$.
	Let
	$$
	\sigma^2 = \bigg\|\sum_{i=1}^n (\vu_i\vu_i^*)^2\bigg\|.
	$$
	Then
	\begin{equation}\label{eq:tf}
	\Disc(\mathbf{u}_1\mathbf{u}_1^*,\ldots,\mathbf{u}_n\mathbf{u}_n^*;\xi_1,\ldots,\xi_n)\leq \sqrt{\frac{n}{d}}\cdot\sigma.
	\end{equation}
Particularly, if $\{\vu_1,\ldots,\vu_n\} \subset \mathbb{C}^d$  form a unit-norm tight frame and $d\leq n\leq 2d-1$, then
	\begin{equation}\label{eq:utf}
	 	\Disc(\mathbf{u}_1\mathbf{u}_1^*,\ldots,\mathbf{u}_n\mathbf{u}_n^*;\xi_1,\ldots,\xi_n)= \sqrt{\frac{n}{d}}\cdot\sigma.
	\end{equation}

\end{theorem}
\begin{remark}
A simple observation is that the upper bound (\ref{eq:tf}) is better than the one in Theorem \ref{th:main} provided $n<9d$.
 According to (\ref{eq:utf}), we can derive the constant $\mathcal{C}_0$ defined in (\ref{eq:c0}) is greater than $\sqrt{2-\frac{1}{d}}$ for any positive integer $d$, which implies that $ \mathcal{C}_0\geq \sqrt{2}$.
\end{remark}

 \subsubsection{Matrix discrepancy for diagonal matrices}
We next show that  one can not remove the $\sqrt{\log d}$ factor in (\ref{eq:tropp}) in general.
\begin{prop}\label{pr:lower}
	Suppose that $\xi_1,\ldots,\xi_n$ are independent Rademacher random variables.
	Let $n\in {\mathbb Z}_{\geq 1}$ and $d=2^n$. Then there exist
	$(1,-1)$-diagonal matrices $\mathbf{A}_1,\ldots,\mathbf{A}_n\in \MC^{d\times d}$ such that
	\[
	\Disc(\mathbf{A}_1,\ldots,\mathbf{A}_n;\xi_1,\ldots,\xi_n)\,\,\geq \,\, n.
	\]
\end{prop}
\begin{proof}
	We  assume that $\vh\in \MC^n$ is a $n$-dimensional vector whose entries are either
	$1$ or $-1$. We use $\vh_1,\ldots,\vh_d$ to denote all these possible vectors where
	$d=2^n$. For $i\in \{1,\ldots,n\}$, we take $\mathbf{A}_i={\rm
		Diag}(\vh_{1,i},\vh_{2,i},\ldots,\vh_{d,i})\in \MC^{d\times d}$, where we use
	$\vh_{k,i}$ to denote the $i$-th entry of $\vh_k$. Assume that
	$\varepsilon=(\varepsilon_1,\ldots,\varepsilon_n)$ is an arbitrary vector in $\{1,-1\}^n$.
	Noting that $\vh_1,\ldots\vh_d$ run over all the vectors in $\{1,-1\}^n$.
	Without loss of generality, we assume that
	$\vh_1=\varepsilon$. Then
	\[
	\bigg\|\sum_{i=1}^n\varepsilon_i\mathbf{A}_i\bigg\|\,\,\geq\,\, \abs{\innerp{\vh_1,\varepsilon}}\,\,\geq\,\, n,
	\]
	which implies the conclusion.
\end{proof}
\begin{remark}
	A simple calculation shows that
	\[
	\sigma^2 = \bigg\|\sum_{i=1}^n \mbox{{\bf Var}}[\xi_i]\mathbf{A}_i^2\bigg\|=n,
	\]
	where $\mathbf{A}_i, i=1,\ldots,n$ and $\xi_i, i=1,\ldots,n$ are defined in the proof of
	Proposition \ref{pr:lower}. Hence, if the factor $\sqrt{\log d}$  in
	(\ref{eq:tropp}) is removed, we have
	\[
	\Disc(\mathbf{A}_1,\ldots,\mathbf{A}_n;\xi_1,\ldots,\xi_n)\,\,\leq \,\,O( \sigma ) \,\,=\,\, O(\sqrt{n}),
	\]
	which contradicts with  Proposition  \ref{pr:lower}.
\end{remark}
%

Motivated by the results above, we present the following conjecture:
\begin{conjecture}
	Suppose that $\xi_1,\ldots,\xi_n$ are independent scalar random variables with finite
	support. Suppose that $\mathbf{A}_1,\ldots,\mathbf{A}_n\in \MC^{d\times d}$ are Hermitian matrices with
	${\rm rank}(\mathbf{A}_i)\leq r, i=1,\ldots,n$. Then we have
	\[
	\Disc(\mathbf{A}_1,\ldots,\mathbf{A}_n;\xi_1,\ldots,\xi_n)\,\,\leq \,\,O( \sqrt{C_{r,n}}\cdot \sigma),
	\]
	where $C_{r,n}:=\max\{\log(r/n),1\}$ and  $\sigma^2 = \big\|\sum_{i=1}^n \mbox{{\bf
			Var}}[\xi_i]\mathbf{A}_i^2\big\|$.
\end{conjecture}

\subsection{Related work}
\label{sec:relatedwork}
 In the past years, one already established many results
about the matrix discrepancy under various assumptions about  matrices $\mathbf{A}_1,\ldots,\mathbf{A}_n$.
We list some results as follows, which show some connections between matrix
discrepancy and other mathematical topics. In this subsection, we assume that $\xi_1,\ldots,\xi_n$ are independent Rademacher random variables.

\subsubsection{Discrepancy minimization: $(0,1)$-diagonal matrices}
Suppose that we have  a set system $\cD=\{D_1,\ldots,D_n\}$ with
$D_i\subseteq\{1,\ldots,d\}$.
 We would like  to find a bi-coloring ~$\chi$: $\{1,\ldots,d\}\to \{\pm 1\}$ such that the most
  imbalance 
  $\max\limits_{i\in\{1,\ldots,n\}}\big|\sum\limits_{j\in D_i}\chi(j)\big|$
  of  $\cD$ is minimized.
  The minimum value is called the \textit{discrepancy} of  the set system $\cD$, denoted
   by $\mbox{disc}(\cD)$.
In a celebrated paper \cite{SpencerSix}, Spencer established the famous six standard
deviations theorem: for any set system ~$\cD$  with $d$ sets and $d$ points, there
exists a bi-coloring $\chi$ such that $\mbox{disc}(\cD) \leq 6\sqrt{d}$. Later,
Gluskin \cite{Gluskin89} and Banaszczyk \cite{bana1988} strengthened Spencer's result
based on deep ideas from convex geometry.
 A well-known conjecture in discrepancy minimization is
Beck and Fiala's conjecture \cite{BF81}, which says
\begin{equation}\label{eq:BF}
\mbox{disc}(\cD)\,\,\leq \,\, O(\sqrt{t}),
\end{equation}
where 
 $t:=\max\limits_{j\in\{1,2,\ldots,d\}} \sum\limits_{i=1}^d\delta_{D_i}(j)$.
Here, we set
\begin{equation*}
\delta_{D_i}(j):=\left\{
        \begin{array}{cl}
        1, & j\in D_i\\
          0, & \mbox{else}
        \end{array}
      \right..
\end{equation*}
One can reformulate Spencer's result with the language of  the matrix discrepancy.
For $i\in \{1,\ldots,d\}$, set
\begin{equation}\label{eq:Ai}
\mathbf{A}_i={\rm Diag}(\delta_{D_1}(i),\ldots,\delta_{D_d}(i)),
\end{equation}
which is a $d\times d$ diagonal matrix.  Under this setting, Spencer's result is
equivalent to
\[
\Disc(\mathbf{A}_1,\ldots,\mathbf{A}_d; \xi_1,\ldots,\xi_d)\,\,\leq\,\, 6\sqrt{d}.
\]
Moreover, Beck and Fiala's conjecture, i.e., (\ref{eq:BF}), is equivalent to
\[
\Disc(\mathbf{A}_1,\ldots,\mathbf{A}_d; \xi_1,\ldots,\xi_d)\,\, \leq \,\, O(\max_i \sqrt{{\rm Tr}(\mathbf{A}_i)}).
\]
 Finally, we would like to mention that, for the
general symmetric matrices, Meka made the following interesting conjecture
 \cite{meka2014}:
\begin{conjecture}[\cite{meka2014}]
Assume that $\xi_1,\ldots,\xi_d$ are independent Rademacher variables.
 For any symmetric
matrices $\mathbf{A}_1,\ldots, \mathbf{A}_d\in \mathbb{R}^{d\times d }$  with $\|\mathbf{A}_i\|\leq 1$, one has
\[
 \Disc(\mathbf{A}_1,\ldots,\mathbf{A}_d; \xi_1,\ldots,\xi_d)\,\, \leq\,\, O(\sqrt{d}).
 \]
\end{conjecture}

\subsubsection{ Kadison-Singer problem: rank-$1$ matrices whose summation equals to an identity matrix}\label{subsec:KS}
Suppose that
 $\vu_1,\ldots,\vu_n\in\mathbb{C}^d$ with $\sum_{i=1}^n \vu_i\vu_i^*=\mathbf{I}$.
Set $\delta := \max_{i}\|\vu_i\vu_i^*\|$. In  \cite{Marcus2015InterlacingII}, Marcus,
Spielman and  Srivastava resolved Kadison-Singer problem. Their result implies that there
exists a partition $\{T_1,T_2\}$ of $\{1,2,\ldots,n\}$ such that
\[
\bigg\|\sum_{i\in T_j} \vu_i\vu_i^*\bigg\|\,\, \leq \,\, \left(\frac{1}{\sqrt{2}}+\sqrt{\delta}\right)^2,\quad
j\in \{1,2\},
\]
which is equivalent to
\begin{equation}\label{eq:KS}
\Disc(\vu_1\vu_1^*,\ldots,\vu_n\vu_n^*; \xi_1,\ldots,\xi_n)\leq
2(\sqrt{2\delta}+\delta).
\end{equation}
 In \cite{bownik2019improved}, Bownik,  Casazza, Marcus, and  Speegle improved
 (\ref{eq:KS}) by employing the theory of mixed discriminant,
   a multilinear generalization of the determinant function. Particularly, they showed that
\begin{equation}\label{eq:IKS}
\Disc(\vu_1\vu_1^*,\ldots,\vu_n\vu_n^*; \xi_1,\ldots,\xi_n)\leq 2\sqrt{2}(\sqrt{1-2\delta})\sqrt{\delta}
\end{equation}
for  vectors $\vu_1,\ldots,\vu_n\in\mathbb{C}^d$ with $\sum_{i=1}^n \vu_i\vu_i^*=\mathbf{I}$
and $\delta = \max_{i}\|\vu_i\vu_i^*\|<1/4$.

In \cite{coh2016}, Cohen adapted the proof in \cite{Marcus2015InterlacingII} to generalize the Kadison-Singer problem to positive semidefinite matrices with arbitrary ranks, which implies that
\begin{equation}
	\Disc(\mathbf{A}_1,\ldots,\mathbf{A}_n; \xi_1,\ldots,\xi_n)\,\,\leq\,\, 2(\sqrt{2\delta}+\delta),
\end{equation}
where  $\mathbf{A}_1,\ldots,\mathbf{A}_n$ are positive semidefinite matrices satisfying
$\sum_{i=1}^n\mathbf{A}_i=\mathbf{I}$ and ~$\delta = \max_i \text{Tr}(\mathbf{A}_i)$.
 In \cite{branden2018hyperbolic}, Br{\"a}nd{\'e}n extended the result into the realm of hyperbolic polynomials and slightly improved Cohen's result,  where   the determinant and matrices are replaced with a hyperbolic polynomial and   elements in the corresponding  hyperbolic cone  respectively.

\subsection{Organization} This paper is  organized as follows.
 In Section \ref{section:Preliminaries},
we introduce some notations and lemmas which are used in our proof.
Section \ref{subsec:reviewofKLS} introduces the  techniques of  mixed discriminant as well as presents a sketch of the proof of Theorem \ref{th:main}.
The proofs of main theorems are given in Sections  \ref{sec:proof} and \ref{sec:schatten}, respectively. In Section \ref{sec:conc}, we extend the matrix discrepancy to Schatten $p$-norm case and present a matrix discrepancy bound under this setting.



\section{Preliminaries}\label{section:Preliminaries}

\subsection{Notations and lemmas}
For a vector $\vv\in \MC^d$, we use $\|\vv\|$ to denote its Euclidean $2$-norm. We use $\R^d_{\geq 0}$ to denote the set of nonnegative vectors in $\R^d$. For
a matrix $\mathbf{M}\in \MC^{d\times d}$, we use $\|\mathbf{M}\|=\max\limits_{\|\vx\|=1}\|\mathbf{M}\vx\|$ to
denote its spectral norm.
We use ${\rm Tr}(\mathbf{M})$ to denote its trace,   and ${\rm adj}(\mathbf{M})$ to
denote its adjugate matrix, i.e. ${\rm adj}(\mathbf{M}) = \big((-1)^{i+j}\mathbf{M}_{ji}\big)_{1\leq
i,j\leq d}$, where $\mathbf{M}_{ji}$ is the $(j,i)$-minor of $\mathbf{M}$, the determinant of the
$(d-1) \times (d-1)$ matrix obtained by removing row $j$ and column $i$ of $\mathbf{M}$.

 We
write $\partial_{z_i}$ to indicate the partial differential
$\partial/\partial_{z_i}$.
 We say that a univariate polynomial is \emph{real-rooted} if all of its
coefficients and roots are real.
 For a real rooted polynomial $p(x)$, we use $\lambda_{\max}(p)$   to denote the largest root of $p(x)$.

We next introduce two lemmas, which are useful in our proofs.
 \begin{lemma}[Jacobi's Formula]\label{lemma:dermaxfun}  Let  $\mathbf{A}$  be a differentiable
 map form the real numbers to ~$\mathbb{C}^{d\times d}$ and $\mathbf{A}(t_0)$ be an invertible matrix. Then
	$$
\frac{\text{d}}{\text{dt}}\bigg|_{t=t_0}\det[\mathbf{A}(t)]=\det[\mathbf{A}(t_0)]{\rm Tr}\Bigg(\mathbf{A}(t_0)^{-1}\frac{\text{d}\mathbf{A}(t))}{\text{dt}}\bigg|_{t=t_0}\Bigg).
$$
\end{lemma}

\begin{lemma}[Matrix Determinant Lemma]\label{lemma:maxdet}
Suppose that $\mathbf{A} \in \MC^{d\times d}$ and $\vu,\vv\in \mathbb{C}^d$. Then
$$
	\det[\mathbf{A}+\vu\vv^*]=\det[\mathbf{A}] + \vv^* {\rm adj}(\mathbf{A}) \vu.
$$

\end{lemma}

\subsection{Interlacing families} 
In this subsection, we introduce interlacing families of polynomials (see
\cite{marcus2015interlacingI,Marcus2015InterlacingII}), which  is an effective tool to show the existence of some combinatorial objects.

\begin{definition}[\cite{marcus2015interlacingI}]
We say that a real rooted polynomial  $g(x)=\alpha_0 \prod_{i=1}^{n-1}(x-\alpha_i)$
\textit{interlaces} a real rooted polynomial $p(x)=\beta_0\prod_{i=1}^n(x-\beta_i)$
if
\[
\beta_1\leq\alpha_1\leq\beta_2\leq\alpha_2\leq\cdots\leq\alpha_{n-1}\leq\beta_n.
\]
We say that  polynomials  $\{p_j\}_{j=1}^k$ have a \textit{common interlacing} if
there exists a polynomial $g$ such that ~$g$ interlaces $p_j$ for each $j$.
\end{definition}

\begin{definition}[\cite{marcus2015interlacingI}]
Let $\cS_1,\ldots,\cS_n$ be finite sets. For every assignment $(s_1,\ldots,s_n)\in
\cS_1\times\cdots\times \cS_n$, let $p_{s_1,\ldots,s_n}(x)$ be a real-rooted degree
$d$ polynomial with positive leading coefficient. For a partial assignment
$(s_1,\ldots,s_k)\in \cS_1\times\cdots\times \cS_k$ with $k<n$, define
$$
p_{s_1,\ldots,s_k}(x):=\sum_{s_{k+1}\in \cS_{k+1},\ldots,s_n\in \cS_n}p_{s_1,\ldots,s_k,s_{k+1},\ldots,s_n}(x)
$$
as well as
$$
p_{\emptyset}(x):=\sum_{s_{1}\in \cS_{1},\ldots,s_n\in \cS_n}p_{s_1,\ldots,s_n}(x).
$$

We say the polynomials $\{p_{s_1,\ldots,s_n}: (s_1,\ldots,s_n)\in \cS_1\times \cdots
\times \cS_n\}$ form an \emph{interlacing family} if for all $k=0,\ldots,n-1$ and all
$(s_1,\ldots,s_k)\in \cS_1\times\cdots\times \cS_k$, the polynomials $
\{p_{s_1,\ldots,s_k,t}\}_{t\in \cS_{k+1}} $ have a common interlacing.
\end{definition}

 The following lemma is necessary for our argument.
\begin{lemma}[{\cite[Theorem 3.4]{Marcus2015InterlacingII}}]
\label{lemma:propertiesofinterfamily} Let $\cS_1,\ldots,\cS_n$ be finite sets and let
$\{p_{s_1,\ldots,s_n}:(s_1,\ldots,s_n)\in \cS_1\times\cdots\times \cS_n\}$
be an
interlacing family.  Then there exist some $(s_1,\ldots,s_n)\in \cS_1\times\cdots\times
\cS_n$ such that
$$
\lambda_{\max}(p_{s_1,\ldots,s_n})\leq \lambda_{\max}(p_{\emptyset}).
$$
\end{lemma}

\section{Mixed discriminants}\label{subsec:reviewofKLS}
This section aims to introduce the mixed discriminant techniques which play an important role in the proof of Theorem \ref{th:main}. Let us begin with  reviewing  the proof of Kyng-Luh-Song  \cite{kyng2019four}.

\subsection{Review of the techniques by Kyng-Luh-Song}
In \cite{kyng2019four},
Kyng, Luh and Song exploited the method of interlacing family to simultaneously control the largest and smallest eigenvalues of the matrices.
%

 Let $\xi_1,\ldots,\xi_n$ be independent  random variables  with finite support. 
 For $j=1,\ldots,n$, we use $\cS_j$ to denote the support of $\xi_j$.
  For any  $(\varepsilon_1,\ldots,\varepsilon_n)\in \mathcal{S}_1\times\cdots\times \mathcal{S}_n$ and  vectors
  $\vv_1,\ldots,\vv_n$ in $\mathbb{C}^d$, we  define
\begin{equation}\label{eq:thechildren}
p_{\varepsilon_1,\ldots,\varepsilon_n}(x):
=\prod_{j=1}^n \mathbb{P}(\xi_j=\varepsilon_j) \det\bigg[x^2\mathbf{I}-\bigg(\sum_{i=1}^n\big(\mathbb{E}[\xi_i]
-\varepsilon_i\big)
  \vv_i\vv_i^*\bigg)^2\bigg]
\end{equation}
  and
 \begin{equation}
  \label{f_emptydef}
  p_{\emptyset}(x) :=\mathop{\mathbb{E}}\limits_{\xi_1,\ldots,\xi_n}\det\bigg[x^2\mathbf{I}-\bigg(\sum_{i=1}^n
  \big(\mathbb{E}[\xi_i]-\xi_i\big)
  \vv_i\vv_i^*\bigg)^2\bigg]=\sum\limits_{\varepsilon_{1}\in\mathcal{S}_1,\ldots,\varepsilon_n\in\mathcal{S}_n}p_{\varepsilon_1,\ldots,\varepsilon_n}(x).
\end{equation}
It follows from \eqref{eq:thechildren} that for any choice of
$\varepsilon_1,\ldots,\varepsilon_n$,
\[
\bigg\|\sum\limits_{i=1}^n\mathbb{E}[\xi_i]\vv_i\vv_i^*-\sum\limits_{i=1}^n\varepsilon_i\vv_i\vv_i^*\bigg\|
=\lambda_{\max}( p_{\varepsilon_1,\ldots,\varepsilon_n}).
\]
 Thus, to prove Theorem \ref{th:main}, it is enough to show that
 there exists  a  choice of
$(\varepsilon_1,\ldots,\varepsilon_n)\in(\cS_{1}\times\cdots\times \cS_n)$ so that
$\lambda_{\max}(p_{\varepsilon_1,\ldots,\varepsilon_n})\leq 3\cdot\delta$.

The following lemma shows that	the polynomials
$\{p_{\varepsilon_1,\ldots,\varepsilon_n}:(\varepsilon_1,\ldots,\varepsilon_n)\in \mathcal{S}_1\times\cdots\times \mathcal{S}_n\}$ defined in (\ref{eq:thechildren})
   form an interlacing family.
 \begin{lemma}[{\cite[Proposition 3.3]{kyng2019four}}]
 \label{lemma:defineinterfaimily}
 	The polynomials $\{p_{\varepsilon_1,\ldots,\varepsilon_n}:(\varepsilon_1,\ldots,\varepsilon_n)\in \mathcal{S}_1\times\cdots\times \mathcal{S}_n\}$   defined in (\ref{eq:thechildren})
   form an interlacing family and
\begin{equation}\label{eq:pemp}
  p_{\emptyset}(x) =\prod_{i=1}^n\bigg(1-\frac{1}{2}\partial_{z_i}^2\bigg)\bigg{|}_{z_i=0}
Q(x,z_1,\ldots,z_n),
\end{equation}
where
\begin{equation}\label{eq:quadratic_poly}
Q(x,z_1,\ldots,z_n)=\det\bigg[x\mathbf{I}+\sum_{i=1}^n z_i\tau_i \vv_i \vv_i^{*}\bigg]^2
\end{equation}
and
 $\tau_i = \sqrt{\mbox{{\bf Var}}[\xi_i]}$, $i=1,\ldots,n$.
\end{lemma}

According to Lemma \ref{lemma:propertiesofinterfamily},  the largest root of
the individual polynomial ~$p_{\varepsilon_1,\ldots,\varepsilon_n}$ is related to the expected
polynomial $p_{\emptyset}$, which leads us to   estimate $\lambda_{\max}(p_{\emptyset})$.
To bound the largest root of $p_{\emptyset} $,
 Kyng, Luh and Song \cite{kyng2019four} utilized
 the {\em barrier function arguments} developed in \cite{Marcus2015InterlacingII,anari2014}. 

A key observation in this paper is the following lemma which says that
$Q(x,z_1,\ldots,z_n)$ is quadratic with respect to each $z_i$. Hence, we can  use the
technology developed in \cite{bownik2019improved} to obtain a better upper bound of
$\lambda_{\max}(p_{\emptyset})$.
\begin{lemma}\label{lemma:Qisquadratic}
The multivariate polynomial $Q(x,z_1,\ldots,z_n)$ defined in
(\ref{eq:quadratic_poly}) is   quadratic with respect to each $z_i$.
\end{lemma}
\begin{proof}
According to (\ref{eq:quadratic_poly}), $Q(x,z_1,\ldots,z_n)$ is a multivariate
polynomial in term of $x,z_1,\ldots,z_n$.  Set $\mathbf{A}_i:= x\mathbf{I}+\sum_{j\neq i}z_j\tau_j\vv_j\vv_j^*$ for $1\leq i\leq n$.
Then we have
	\begin{equation*}
		\begin{aligned}
			Q(x,z_1,\ldots,z_n) &= \det[\mathbf{A}_i+z_i\tau_i\vv_i\vv_i^*]^2\\
								&= \big(\det[\mathbf{A}_i] + \vv_i^*{\rm adj}(\mathbf{A}_i)(z_i\tau_i\vv_i)\big)^2\\
								&= \big(\det[\mathbf{A}_i]+z_i(\tau_i\vv_i^*{\rm adj}(\mathbf{A}_i)\vv_i)\big)^2,
		\end{aligned}
	\end{equation*}
		where the second equality follows from Lemma \ref{lemma:maxdet}.
	 It follows
that $Q(x,z_1,\ldots,z_n)$ is quadratic with respect to $z_i$.
\end{proof}




\subsection{Barrier functions and mixed discriminant}
In this subsection, we introduce  the definitions of barrier functions and  mixed
discriminant, which are useful  in estimating $\lambda_{\max}(p_\emptyset)$. Let us
begin with the definition of real stability.

\begin{definition}
A   polynomial
$p\in \mathbb{R}[z_1,\ldots,z_n]$  is \emph{real stable} if
~$p(z_1,\ldots,z_n)\neq 0$ for  all $(z_1,\ldots,z_n)\in\mathbb{C}^n$ with $\mbox{{\bf Im}}(z_i)>0$ and $1\leq i\leq n$.
\end{definition}

The following lemma provides an important class of real stable polynomials.
\begin{lemma}[{\cite[Proposition 2.4]{borcea2008applications}}]\label{le:sta1}
If $\mathbf{A}_1,\ldots,\mathbf{A}_n\in\mathbb{C}^{d\times d}$  are positive semidefinite symmetric matrices and $\mathbf{B}\in\mathbb{C}^{d\times d}$ is a Hermitian matrix, then the
polynomial
\[
p(z_1,\ldots,z_n):=\det \bigg(\sum_{i=1}^n z_i\mathbf{A}_i+\mathbf{B}\bigg)
\]
is real stable.
\end{lemma}
The real stability can be preserved by some differential operators.
\begin{lemma}[{\cite[Theorem 1.3]{borcea2010multivariate}}]\label{le:sta2}
If $p\in \R[z_1,\ldots,z_n]$ is real stable, then for any $c>0$, the polynomial
\[
(1-c\partial_{z_i}^2)p(z_1,\ldots,z_n)
\]
is real stable for all $1\leq i \leq n$.
\end{lemma}

We also need the following definition.

\begin{definition}[\cite{Marcus2015InterlacingII}]
\label{Def:above the roots} Let $p(z_1,\ldots,z_n)$ be a multivariate polynomial. We
say that   ~$\vz\in\mathbb{R}^n$ is \textit{above the roots} of $p$ if
\begin{equation*}
 p(\vz+\vt)>0 \quad \text{for all} \quad 
 \vt\in\mathbb{R}^n_{\geq 0}.
\end{equation*}
We use $\mbox{{\bf Ab}}_p$ to denote the set of  points  which are above the roots of
$p$.
\end{definition}

Now we are ready to introduce  the definition of barrier functions. 	
\begin{definition}[\cite{batson2012twice-ramanujan, Marcus2015InterlacingII}]
 Given a real stable polynomial $p(z_1,\ldots,z_n)$ and a point ~$\vz\in \mbox{{\bf Ab}}_p$,
 the \textit{barrier function} of $p$ in direction $i$ at ${\vz}=(z_1,\ldots,z_n)$ is defined as
$$\Phi_p^i({\mathbf z}) = \frac{\partial_{z_i}p(\vz)}{p(\vz)}.$$
\end{definition}

The barrier function of a polynomial $p\in\mathbb{R}[z_1,\ldots,z_n]$ at a point $\mathbf{z}=(z_1,\ldots,z_n)\in\mathbf{Ab}_p$ can be viewed as  quantization of the point $ \mathbf{z}=(z_1,\ldots,z_n) $ lying above the roots of $p$, see the physical explanation for the univariate case in \cite{batson2012twice-ramanujan}. The purpose of  barrier function argument is to  study the relationship between $\mathbf{Ab}_p$ and $\mathbf{Ab}_{F(p)}$  with the help of barrier function, where $F$ is a differential operator. In our problem, the differential operator is $F=1-\frac{1}{2}\partial_{z_i}^2$, see Lemma \ref{lemma:defineinterfaimily}. 
	
In the remainder of this section, we introduce properties of barrier function and connections with the mixed discriminant. The following lemma shows the monotonicity of barrier function in each direction.
\begin{lemma}[{\cite[Lemma 5.8]{Marcus2015InterlacingII}}]\label{le:monoto}
Assume that $p(z_1,\ldots,z_n)$ is a real stable polynomial and  ~$\vz\in \mbox{{\bf
Ab}}_p$. Then
\[
\Phi_p^i(\vz+\vt)\,\,\leq\,\, \Phi_p^i(\vz)
\]
holds for any  $\vt\in\R_{\geq 0}^n$
 and $i\in \{1,\ldots,n\}$.
\end{lemma}

 We also need the following result about barrier functions.
\begin{lemma}[{\cite[Lemma 4.5]{anari2014}}]\label{le:secbar}
Assume that $p(z_1,\ldots,z_n)$ is a real stable polynomial and  ~$\vz\in \mbox{{\bf Ab}}_p$. Then
\[
\frac{\partial_{z_i}^2p}{p}(\vz)\leq\Phi_p^i(\vz)^2
\]
holds for each $i\in\{1,\ldots,n\}$.
\end{lemma}

For a univariate quadratic  polynomial,  we have the following result related to its barrier function.
\begin{lemma}[{\cite[Lemma 3.8]{bownik2019improved}}]
\label{Lemma:quadraticbarriernonincreasing}
 Suppose that $s(x)$ is a univariate, quadratic polynomial with positive leading coefficient.  Let $\Phi_s(x) = \frac{s'(x)}{s(x)}$ be its barrier function. Then
\[
f(x)=x-\frac{2}{\Phi_s(x)}
\]
is a nonincreasing function on $ \mbox{{\bf Ab}}_s$.
\end{lemma}

The following lemma, which is essential for our proof, shows the behavior of barrier
functions when applying the differential operator $(1-\frac{1}{2}\partial_x^2)$ to a
bivariate quadratic polynomial. We postpone its proof to the end of this subsection.

\begin{lemma}\label{Lemma:QuadraticBarrier}
Assume that  $p(x,y)$ is a bivariate real stable polynomial and  it is quadratic with
respect to $x$. Suppose that $(x_0,y_0)\in \mbox{{\bf Ab}}_p$ and
$(x_0+\delta,y_0)\in \mbox{{\bf Ab}}_{(1-\frac{1}{2}\partial_x^2)p}$. The followings
hold:
\begin{enumerate}[ (i)]
\item If $\delta=1$, then
    $\Phi_{(1-\frac{1}{2}\partial_x^2)p}^y(x_0+\delta,y_0)\leq\Phi_p^y(x_0,y_0)$.
\item If $\delta\in(0,1)$  and the barrier function  satisfies
    $\Phi_p^x(x_0,y_0)\leq \frac{\delta}{1-\delta^2},$ then
 $\Phi_{(1-\frac{1}{2}\partial_x^2)p}^y(x_0+\delta,y_0)\leq\Phi_p^y(x_0,y_0)$.
\end{enumerate}
\end{lemma}

Motivated by the method developed in \cite{bownik2019improved}, we employ the mixed
discriminant to prove Lemma \ref{Lemma:QuadraticBarrier}.
 We next recall  the definition of mixed discriminant as well
as its properties (see \cite{bownik2019improved}).
\begin{definition}\label{mixdis}
Let $\mathbf{X}_1,\ldots,\mathbf{X}_d\in\mathbb{R}^{d\times d}$. The  \textit{mixed discriminant}  of
$\mathbf{X}_1,\ldots,\mathbf{X}_d$ is defined as
$$
D(\mathbf{X}_1,\ldots,\mathbf{X}_d):=\frac{\partial^d}{\partial t_1\ldots\partial t_d}\det\bigg[\sum_{i=1}^dt_i\mathbf{X}_i\bigg].
$$
\end{definition}
For convenience, given a matrix $X$, we set $X[k]:=(\underbrace{X,\ldots,X}_{k})$.
Let $\mathbf{X}_1,\ldots,\mathbf{X}_k\in\mathbb{R}^{d\times d}$   and we set
$$
\tilde{D}(\mathbf{X}_1,\ldots,\mathbf{X}_k):=\frac{D(\mathbf{X}_1,\ldots,\mathbf{X}_k,\mathbf{I}[d-k])}{(d-k)!}.
$$
The following lemma shows an interesting property of $\tilde{D}$.
\begin{lemma}[{\cite[Theorem 1.1]{artsteinavidan2014remarks}}]\label{le:Dtil}
Assume that $\mathbf{X}_1,\mathbf{X}_2\in\mathbb{R}^{d\times d}$ are  positive semidefinite matrices with $d\geq 2$.
Then
\[
\tilde{D}(\mathbf{X}_1)\tilde{D}(\mathbf{X}_2) \geq \tilde{D}(\mathbf{X}_1,\mathbf{X}_2).
\]
\end{lemma}

For a bivariate polynomial $p\in\mathbb{R}[x,y]$, we can find a determinantal representation for the polynomial $p$ if $p$ is real stable, which follows by the Lax conjecture proved in \cite{LPR05} by the results in \cite{HV07}.
\begin{theorem}[{\cite[Theorem 1.13]{borcea2010multivariate}}]\label{thm:detrep}
	Suppose that  $p(x,y)$ is a bivariate real stable polynomial of degree $d$. Then there exsit positive semidefinite matrices $\mathbf{A},\mathbf{B}\in\mathbb{R}^{d\times d}$ and symmetric matrix $\mathbf{C}\in\mathbb{R}^{d\times d}$ such that
	\[
		p(x,y)=\pm \det[x\mathbf{A}+y\mathbf{B}+\mathbf{C}]
	\]
\end{theorem}	
\begin{remark}\label{rem:detrep}
 	Futhermore, if $(x_0,y_0)\in \mathbf{Ab}_p$, it is observed that $p(x,y)= \det[x\mathbf{A}+y\mathbf{B}+\mathbf{C}]$ and the matrix  $\mathbf{M}:=x_0\mathbf{A}+y_0\mathbf{B}+\mathbf{C}$ is positive definite in \cite[Corollary 3.4]{bownik2019improved}.
\end{remark}
 The following lemma
characterizes the relationship between the mixed determinant and the barrier function of a bivariate polynomial under  transformation of a differential operator.

\begin{lemma}[{\cite[Corollary 3.5, Lemma 3.6]{bownik2019improved}}]
\label{Lemma:barrierequvimixed}
 Suppose that  $p(x,y)$ is a bivariate real stable polynomial and
    $(x_0,y_0)\in {\bf Ab}_p$.  Let $F=\sum_{i=0}^na_i\partial_x^i$ be  a differential operator  with real
    coefficients $\{a_i\}_{i=1}^n$ and let ~$q(x,y)=F(p(x,y))$.
    Set
    \begin{equation}\label{eq:ABL}
    	\hat{\mathbf{A}}:=\mathbf{M}^{-\frac{1}{2}}\mathbf{A}\mathbf{M}^{-\frac{1}{2}},\quad
    	\hat{\mathbf{B}}:=\mathbf{M}^{-\frac{1}{2}}\mathbf{B}\mathbf{M}^{-\frac{1}{2}} \quad \mbox{and} \quad
    	\hat{\mathbf{L}}:=\hat{\mathbf{A}}^{\frac{1}{2}}\hat{\mathbf{B}}\hat{\mathbf{A}}^{\frac{1}{2}},
    \end{equation}
	 where the matrices $\mathbf{A},\mathbf{B}$ and $\mathbf{M}$ are defined in Theorem \ref{thm:detrep} and Remark\ \ref{rem:detrep}.
 Then  the followings hold:
\begin{enumerate}[(i)]
\item  
The barrier function and the mixed discriminant is related as follows:
\[
\Phi_p^x(x_0,y_0)={\tilde D}(\hat{\mathbf{A}}).
\]
\item 
     If $(x_0,y_0) \in \mbox{{\bf Ab}}_p\cap \mbox{{\bf Ab}}_q$, then
$$\Phi_q^y(x_0,y_0) \leq \Phi_p^y(x_0,y_0)$$ if and only if
$$\sum_{i=1}^nia_i\tilde{D}(\hat{\mathbf{A}}[i-1],\hat{\mathbf{L}})\geq 0.$$

\end{enumerate}

\end{lemma}

Now, we are ready to prove Lemma \ref{Lemma:QuadraticBarrier}.

\begin{proof}[Proof of Lemma \ref{Lemma:QuadraticBarrier}]
  Our aim is  to prove
  \begin{equation}\label{eq:technical_lemma}
    \Phi_q^y(x_0,y_0)\leq\Phi_p^y(x_0,y_0),
  \end{equation}
where    $q(x,y)=(1-\frac{1}{2}\partial_x^2)p(x+\delta,y)$.
 Since $p(x,y)$ is quadratic in $x$,  by Taylor expansion, we have
$$
\begin{aligned}
q(x,y)&=p(x+\delta,y)-\frac{1}{2}\partial_x^2 p(x+\delta,y)\\
   &=(1+\delta\partial_x+\frac{1}{2}\delta^2\partial_x^2)p(x,y)-\frac{1}{2}\partial_x^2p(x,y)\\
  &=(a_0 +a_1 \partial_x + a_2 \partial_x^2)p(x,y),
\end{aligned}
$$
  where $a_0 =1, a_1=\delta>0$
and  $a_2=\frac{1}{2}(\delta^2-1)$. According to   Lemma
\ref{Lemma:barrierequvimixed},  (\ref{eq:technical_lemma}) is equivalent to
\begin{equation}
\label{proofl-1}
a_1\widetilde{D}(\hat{\mathbf{L}})+2a_2\widetilde{D}(\hat{\mathbf{A}},\hat{\mathbf{L}})\geq 0.
\end{equation}
Here, $\hat{\mathbf{A}}$ and $\hat{\mathbf{L}}$ are defined  in \eqref{eq:ABL}.
According to the  definition of $\hat{\mathbf{L}}$,   we know that $\hat{\mathbf{L}}$ is a positive
semidefinite matrix. The definition of mixed discriminant implies
\begin{equation}\label{eq:Dpos}
\tilde{D}(\hat{\mathbf{L}})=D(\hat{\mathbf{L}},\mathbf{I}[d-1])/(d-1)!={\rm Tr}(\hat{\mathbf{L}})\geq 0.
\end{equation}

We first consider (i). If $\delta =1$, then $a_2=\frac{1}{2}(\delta^2-1)=0$. Hence,
\eqref{proofl-1} follows from \eqref{eq:Dpos}.

We next turn to (ii), i.e. $\delta\in (0,1)$. If $\tilde{D}(\hat{\mathbf{L}})={\rm Tr}(\hat{
\mathbf{L}})=0$, then $\hat{\mathbf{L}}=0$ which implies $\tilde{D}(\hat{\mathbf{A}},\hat{\mathbf{L}})=0$. So
\eqref{proofl-1} holds. We next consider the case where $\tilde{D}(\hat{\mathbf{L}})>0$.

According to Lemma \ref{Lemma:barrierequvimixed}, we have
\[
\widetilde{D}(\hat{\mathbf{A}})=\Phi_p^x(x_0,y_0)\leq\frac{\delta}{1-\delta^2}.
\]
Hence, we obtain that
\[
\frac{\tilde{D}(\hat{\mathbf{A}},\hat{\mathbf{L}})}{\tilde{D}(\hat{\mathbf{L}})}\leq \widetilde{D}(\hat{\mathbf{A}})\leq\frac{\delta}{1-\delta^2},
\]
which implies \eqref{proofl-1}. Here, the first inequality follows from Lemma
\ref{le:Dtil}.
\end{proof}

\section{ Proof of Theorem \ref{th:main}}\label{sec:proof}

To prove Theorem \ref{th:main}, as discussed in Section \ref{subsec:reviewofKLS}, we
need to bound the largest root of $p_{\emptyset}(x)$  defined in
\eqref{eq:pemp}. The following theorem presents an upper bound of
$\lambda_{\max}(p_{\emptyset})$ and we postpone its proof.
  \begin{theorem}\label{lemma:estimatethemaxroot}
 Let
$\vv_1,\ldots,\vv_n\in \mathbb{C}^d$  and $\tau_1,\ldots,\tau_n> 0 $ such that
  \begin{equation}\label{eq:normalized_condition}
   \sum_{i=1}^n\tau_i^2(\vv_i \vv_i^{*})^2\preceq \mathbf{I}.
  \end{equation}
Then the largest root of the polynomial
    \begin{equation}\label{eq:pemptyset}
      p_{\emptyset}(x) :=\prod_{i=1}^n\Bigg(1-\frac{1}{2}\partial_{z_i}^2\Bigg)\Bigg{|}_{z_i=0}\det\bigg[x\mathbf{I}+\sum_{i=1}^n z_i\tau_i \vv_i \vv_i^{*}\bigg]^2
    \end{equation}
is at most $3$.
  \end{theorem}

 Using this theorem, we next present the proof of Theorem \ref{th:main}.
\begin{proof}[Proof of Theorem \ref{th:main}]
Set $\vv_i=\frac{\vu_i}{\sqrt{\sigma}}$ for $i=1,\ldots,n$. Then the vectors $\vv_1,\ldots,\vv_n$
satisfy $\bigg\|\sum\limits_{i=1}^n{\mathbf{Var}}[\xi_i](\vv_i\vv_i^*)^2\bigg\|=1$. Recall that
\[
p_{\varepsilon_1,\ldots,\varepsilon_n}(x)
=\prod_{j=1}^n \mathbb{P}(\xi_j=\varepsilon_j) \det\bigg[x^2\mathbf{I}-\bigg(\sum_{i=1}^n\big(\mathbb{E}[\xi_i]-\varepsilon_i\big)
  \vv_i\vv_i^*\bigg)^2\bigg]
\]
and
\[
\bigg\|\sum\limits_{i=1}^n\mathbb{E}[\xi_i]\vv_i\vv_i^*-\sum\limits_{i=1}^n\varepsilon_i\vv_i\vv_i^*\bigg\|
=\lambda_{\max}( p_{\varepsilon_1,\ldots,\varepsilon_n}).
\]
 According to Lemma \ref{lemma:defineinterfaimily}, $\{
p_{\varepsilon_1,\ldots,\varepsilon_n}:(\varepsilon_1,\ldots,\varepsilon_n)\in \cS_1\times
\cdots\times  \cS_n\}$ form an interlacing family where $\cS_j$ is the support of
$\xi_j$. Combining   Lemmas \ref{lemma:propertiesofinterfamily}, \ref{lemma:defineinterfaimily} and Theorem
\ref{lemma:estimatethemaxroot}, we obtain that there exists
$(\varepsilon_1,\ldots,\varepsilon_n)\in \cS_1\times \cdots\times  \cS_n$ such that
\[
\lambda_{\max}(p_{\varepsilon_1,\ldots,\varepsilon_n})\leq \lambda_{\max}(p_{\emptyset})\leq 3.
\]
Hence, there exists a choice of outcomes $\varepsilon_1,\ldots,\varepsilon_n$ such that
$$
\bigg\|\sum\limits_{i=1}^n\mathbb{E}[\xi_i]\vv_i\vv_i^*-\sum\limits_{i=1}^n\varepsilon_i\vv_i\vv_i^*\bigg\|\leq 3,
$$
which implies
$$
\bigg\|\sum_{i=1}^n\mathbb{E}[\xi_i]\vu_i\vu_i^*-\sum_{i=1}^n\varepsilon_i\vu_i\vu_i^*\bigg\|\leq 3\sigma.
$$
\end{proof}

The rest of this section aims to prove Theorem \ref{lemma:estimatethemaxroot}. 
Our proof adapts the multivariate barrier argument developed  by Marcus, Spielman,
and Srivastava in \cite{Marcus2015InterlacingII}. Recall that
$$
Q(x,z_1,\ldots,z_n)\,\,=\,\,\det\bigg[x\mathbf{I}+\sum_{i=1}^n z_i\tau_i \vv_i \vv_i^{*}\bigg]^2.
$$
We set
\begin{equation}
\label{eq:w_k}
    Q_k(x,z_1,\ldots,z_n)\,\,:=\,\,\prod_{i=1}^k \big(1-\frac{1}{2}\partial_{z_i}^2\big)Q(x,z_1,\ldots,z_n), \ \ k=1,\ldots,n.
\end{equation}
A simple observation is that $Q_n(x,0,\ldots,0)=p_{\emptyset}(x)$.  We set
\[
\delta_i \,\,:=\,\, \tau_i \vv_i^{*} \vv_i, i=1,2,\ldots,n,
\]
and
\[
\vw_k:=(0,\ldots,0,-\delta_{k+1},\ldots,-\delta_n)\in\mathbb{R}^n, \quad k=0,1, \ldots, n-1.
\]
We also set   $\vw_n={\bf 0}\in\mathbb{R}^n$.
According to (\ref{eq:normalized_condition}), we know  $\delta_i\in (0,1]$. Lemma \ref{lemma:Qisquadratic} shows that $Q(x,z_1,\ldots,z_n)$ is quadratic with respect
to each ~$z_i$ and hence  $Q_k$ is also quadratic with respect to each $z_i$.
Combining Lemma \ref{le:sta1} and Lemma \ref{le:sta2}, we obtain that ~$Q_k$ is a
real stable polynomial.


We study the conversion of the barrier function of polynomial transformed by the operator $1-\frac{1}{2}\partial_{z_i}$ as the following lemma.
To state our proof clearly,  we postpone its proof until  the end of this section.
\begin{lemma}\label{Lemma:mainlemma2}
Let $k$ be an integer with $0\leq k\leq n-1$ and $\vw_k$ be defined as in
(\ref{eq:w_k}). Suppose that $\alpha\in\mathbb{R}$ satisfies $(\alpha, \vw_k)\in
\mbox{{\bf Ab}}_{Q_k}$   and   $$\Phi_{Q_k}^{k+1}(\alpha,\vw_k)\leq \delta_{k+1}.$$
Then $(\alpha,\vw_{k+1})\in \mbox{{\bf Ab}}_{Q_{k+1}}$ and
  $$
  \Phi_{Q_{k+1}}^j(\alpha,\vw_{k+1})\leq \Phi_{Q_k}^j(\alpha,\vw_k),
  $$
   where $j$ is an integer with $k+1<j\leq n$.
\end{lemma}

 Now we are ready to prove Theorem \ref{lemma:estimatethemaxroot}.

 \begin{proof}[Proof of Theorem \ref{lemma:estimatethemaxroot}]
   Plugging $(3,\vw_0)$ into $Q$ and noting that   $\sum_{i=1}^n\tau_i^2(\vv_i\vv_i^*)^2\preceq \mathbf{I}$,
    we have
$$
 Q(3,\vw_0)=Q(3,-\delta_1,\ldots,-\delta_n)
  =\det\bigg[3\mathbf{I} -\sum_{i=1}^n\tau_i^2(\vv_i\vv_i^*)^2\bigg]^2
   \geq\det[2\mathbf{I}]^2>0,
$$
which implies that the  initial point  $(3,\vw_0)\in \mbox{{\bf Ab}}_Q$.
 The barrier function  of $Q$ at $(3,\mathbf{w}_0)$ in direction $i$ is
$$
\begin{aligned}
 \Phi_{Q}^i(3,\vw_0)&=\frac{\partial_{z_i}Q}{Q}\Bigg|_{(x,\vz)=(3,\vw_0)}\\
 &=\frac{2\det[x\mathbf{I}+\sum_{j=1}^nz_j\tau_j\vv_j\vv_j^*]
 \cdot\frac{\partial}{\partial z_i}\det[x\mathbf{I}+\sum_{j=1}^nz_j\tau_j\vv_j\vv_j^*]}
 {\det[x\mathbf{I}+\sum_{j=1}^nz_j\tau_j\vv_j\vv_j^*]^2}\Bigg|_{(x,\vz)=(3,\vw_0)}\\
  &=2\Tr \Bigg(\Big(x\mathbf{I}+\sum_{j=1}^n z_j\tau_j\vv_j\vv_j^{*}\Big)^{-1}\tau_i\vv_i\vv_i^{*}\Bigg)\Bigg|_{(x,\vz)=(3,\vw_0)} \\
   &=2\Tr\Bigg(\Big(3\mathbf{I}-\sum_{j=1}^n(\tau_j\vv_j\vv_j^*)^2\Big)^{-1}\tau_i\vv_i\vv_i^*\Bigg)\\
    &\leq 2 \Tr (2^{-1}\tau_i \vv_i \vv_i^{*})=\tau_i\vv_i^*\vv_i=\delta_i,
\end{aligned}
$$
where the third equality follows from Lemma \ref{lemma:dermaxfun}. Hence, we have
   $ \Phi_{Q}^i(3,\vw_0)\leq \delta_i$ for any
$i=1,\ldots,n$. According to  Lemma \ref{Lemma:mainlemma2}, we obtain that
~$(3,\vw_1)\in \mbox{{\bf Ab}}_{Q_1}$ and
$\Phi^i_{Q_1}(3,\vw_1)\leq\Phi^i_{Q}(3,\vw_0)\leq \delta_i$ for any $i=2,\ldots,n$.

Repeating this argument for each $i\in\{2,\ldots,n\}$, we conclude that
\[
(3,\vw_n)=(3,0,\ldots,0)\in \mbox{{\bf Ab}}_{Q_n}.
\]
 The definition of
$Q_n$ implies  $Q_n(x,0,\ldots,0)=p_{\emptyset}(x)$. Hence, we have
$\lambda_{\max}(p_{\emptyset})\leq 3$.
 \end{proof}



  \begin{proof}[Proof of Lemma \ref{Lemma:mainlemma2}]
Firstly, we  will show that $(\alpha,\vw_{k+1})\in \mbox{{\bf Ab}}_{Q_{k+1}}$.
According to Definition \ref{Def:above the roots}, it is sufficient to show that
\begin{equation}
\label{proofl321}
Q_{k+1}((\alpha,\vw_{k+1})+\vt)=(1-\frac{1}{2}\partial_{z_{k+1}}^2)Q_k((\alpha,\vw_{k+1})+\vt)>0
\end{equation}
for any $\vt\in\R_{\geq 0}^{n+1}$.
 Noting that $(\alpha,\vw_{k})\in \mbox{{\bf Ab}}_{Q_k}$ and
  $\vw_{k+1}=\vw_{k}+\delta_{k+1}\ve_{k+1}, \delta_{k+1}>0$, we have
  $(\alpha,\vw_{k+1})\in \mbox{{\bf Ab}}_{Q_k}$. Here, we use $\ve_j,j = 1, \ldots,n$
  to denote a vector in $\mathbb{R}^{n}$ whose $j$-th entry is $1$ and other
  entries are
$0$. So $Q_k((\alpha,\vw_{k+1})+\vt)>0$ for any 
$\vt\in\R_{\geq 0}^{n+1}$.
 We claim
\begin{equation}\label{eq:claims2}
\Phi_{Q_k}^{k+1}(\alpha,w_{k+1})< \sqrt 2.
\end{equation}
Hence, we have
\begin{equation*}
	\begin{aligned}
\frac{\partial_{z_{k+1}}^2Q_k}{Q_k}((\alpha,\vw_{k+1})+\vt)&
\leq
  \big(\Phi_{Q_k}^{k+1}((\alpha,\vw_{k+1})+\vt)\big)^2\\
 &\leq \big(\Phi_{Q_k}^{k+1}(\alpha,\vw_{k+1})\big)^2< 2,
    \end{aligned}
\end{equation*}
which implies (\ref{proofl321}).
 Here, the first line follows from Lemma \ref{le:secbar} 
  and the second line follows from Lemma \ref{le:monoto}.

We still need to prove (\ref{eq:claims2}). Set
\[
s(x)\,\, :=\,\,Q_k(\alpha,\underbrace{0,\ldots,0}_k,x,-\delta_{k+2},\ldots,-\delta_n),
\]
which is quadratic polynomial and real stable. Recall that $(\alpha,w_k),(\alpha,w_{k+1})\in
\mbox{{\bf Ab}}_{Q_k}$. So, we have  $-\delta_{k+1}\in \mbox{{\bf Ab}}_s$  and $0\in
\mbox{{\bf Ab}}_s$. The definition of $s(x)$ also shows that
\begin{equation*}
  \begin{array}{ll}
    \Phi_s(0)=\Phi_{Q_k}^{k+1}(\alpha,\vw_{k+1})
  \end{array}
  \hspace{1em}
  \text{and}
  \hspace{1em}
  \begin{array}{ll}
    \Phi_s(-\delta_{k+1})=\Phi_{Q_k}^{k+1}(\alpha,\vw_k).
  \end{array}
\end{equation*}
We have
\begin{equation*}
\begin{aligned}
\Phi_{Q_k}^{k+1}(\alpha,\vw_{k+1})&=\Phi_s(0)\\
&\leq \frac{2\Phi_s(-\delta_{k+1})}{\delta_{k+1}\Phi_s(-\delta_{k+1})+2}\\
&<\sqrt{2},
\end{aligned}
\end{equation*}
which implies (\ref{eq:claims2}). Here, the second line follows from Lemma
\ref{Lemma:quadraticbarriernonincreasing}, i.e.,
$$
-\delta_{k+1}-\frac{2}{\Phi_s(-\delta_{k+1})}\geq 0-\frac{2}{\Phi_s(0)},
$$
and the third line follows from
$\Phi_s(-\delta_{k+1})=\Phi_{Q_k}^{k+1}(\alpha,w_k)\leq \delta_{k+1}<
\frac{2}{\sqrt 2-\delta_{k+1}} $.

We next prove that $\Phi_{Q_{k+1}}^j(\alpha,w_{k+1})\leq \Phi_{Q_k}^j(\alpha,w_k)$
for $k+1<j \leq n$. Let
\[
p_{k,j}(x,y) =
Q_k(\alpha,0,\cdots,0,\mathop{\underline{x}}\limits_{k+1},-\delta_{k+2},\cdots,-\delta_{j-1},
\mathop{\underline{y}}\limits_{j},-\delta_{j+1},\ldots,-\delta_{n}),
\]
which is quadratic with respect to $x$
and real stable. According to
$(\alpha,\vw_{k})\in \mbox{{\bf Ab}}_{Q_{k}}$, we have
  $(x_0,y_0)\in \mbox{{\bf Ab}}_{p_{k,j}}$ where $(x_0,y_0)=(-\delta_{k+1},-\delta_{j})\in
  \mathbb{R}^2$.
   By the above the discussion, we know that
$(x_0+\delta_{k+1},y_0)\in \mbox{{\bf
Ab}}_{(1-\frac{1}{2}\partial_x^2)p_{k,j}(x,y)}$. So, we have
 \begin{equation*}
    \begin{array}{ll}
      \Phi_{p_{k,j}}^x(x_0,y_0)=\Phi_{Q_k}^{k+1}(\alpha,w_k)\leq\delta_{k+1}
    \end{array}
    \hspace{1em}
    \text{and}
    \hspace{1em}
    \begin{array}{ll}
      \Phi_{p_{k,j}}^x(x_0+\delta_{k+1},y_0)=\Phi_{Q_{k}}^{k+1}(\alpha,w_{k+1}).
    \end{array}
       \end{equation*}
Therefore, to show  $\Phi_{Q_{k+1}}^j(\alpha,w_{k+1})\leq \Phi_{Q_k}^j(\alpha,w_k)$
for $k+1<j\leq n$, it is sufficient to prove that
   \begin{equation}\label{eq:zuihou}
   \Phi^y_{(1-\frac{1}{2}\partial_x^2)p_{k,j}}(x_0+\delta_{k+1},y_0)\leq \Phi_{p_{k,j}}^y(x_0,y_0).
   \end{equation}
   According to Lemma \ref{Lemma:QuadraticBarrier},
   (\ref{eq:zuihou}) holds for $\delta_{k+1} = 1$. For the case where  $\delta_{k+1}\in(0,1)$,
   we have
\begin{equation}
\label{eq:keyeqinlemma2}
  \Phi_{p_{k,j}}^x(x_0,y_0)\leq \delta_{k+1} \leq\frac{\delta_{k+1}}{1-\delta_{k+1}^2}.
\end{equation}
Combining (\ref{eq:keyeqinlemma2}) and Lemma \ref{Lemma:QuadraticBarrier},
we arrive at the conclusion.
  \end{proof}

\section{Proof of Theorem \ref{thm:sharp}}\label{sec:schatten}
This section aims to prove Theorem \ref{thm:sharp}. We first introduce the following lemma.
\begin{lemma}
\label{lemma-2021-5-5}
Suppose that  $\{\vu_1,\ldots,\vu_n\} \subset \mathbb{C}^d$  is  a tight frame in ${\mathbb C}^d$ with frame bound $C$.  Then
$$
\bigg\|\sum\limits_{i=1}^n(\vu_i\vu_i^*)^2\bigg\|\geq C^2 \cdot \frac{d}{n}.
$$
\end{lemma}
\begin{proof} Since $\sum\limits_{i=1}^n\vu_i\vu_i^*=C\cdot\mathbf{I}$, we have
\begin{equation}\label{proof-12-12}
\sum\limits_{i=1}^n\|\vu_i\|^2=\mbox{Tr}\bigg(\sum\limits_{i=1}^n\vu_i\vu_i^*\bigg)=C\cdot d.
\end{equation}
Noting that
$$
\begin{array}{ll}
\mbox{Tr}\bigg(\sum\limits_{i=1}^n(\vu_i\vu_i^*)^2\bigg)&=
\mbox{Tr}\bigg(\sum\limits_{i=1}^n\|\vu_i\|^2\vu_i\vu_i^*\bigg)\\
&=\sum\limits_{i=1}^n \|\vu_i\|^4\\
&\geq \frac{1}{n}\cdot \bigg(\sum\limits_{i=1}^n \|\vu_i\|^2\bigg)^2\\
&= C^2\cdot \frac{d^2}{n},
\end{array}
$$
where the inequality follows from Cauchy-Schwarz inequality. 
Hence
$$
\bigg\|\sum\limits_{i=1}^n(\vu_i\vu_i^*)^2\bigg\|=
\lambda_{\max}\bigg(\sum\limits_{i=1}^n(\vu_i\vu_i^*)^2\bigg)\geq \frac{1}{d}{\mbox{Tr}\big(\sum\limits_{i=1}^n(\vu_i\vu_i^*)^2\big)}\geq C^2\cdot\frac{d}{n},
$$
where $\lambda_{\max}$ denotes the largest eigenvalue of a  matrix.
\end{proof}


We also need the following lemma.
\begin{lemma}
	\label{lemma-lower-321-1}
Suppose that $\vu_1,\ldots,\vu_n\in\mathbb{C}^d$  satisfy $\sum\limits_{i=1}^n \vu_i\vu_i^* = \frac{n}{d}\cdot \mathbf{I}.$ 
If $n\leq 2d-1$, then for any $\varepsilon_1,\ldots,\varepsilon_n\in\{-1,1\}$, we have
	$$
	\bigg\|\sum\limits_{i=1}^n \varepsilon_i \vu_i\vu_i^*\bigg\|=\frac{n}{d}.
	$$
\end{lemma}
\begin{proof}
To prove the conclusion, it is enough  to show that for any partition $\cS_1,\cS_2$ of $\{1,\ldots,n\}$, we have
	$$
	\bigg\|\sum\limits_{i\in \cS_2}  \vu_i\vu_i^*-\sum\limits_{j\in \cS_1}  \vu_j\vu_j^*\bigg\|=\frac{n}{d}.
	$$
	Without loss of generality, we assume $|\cS_1|\leq|\cS_2|$. The  $n\leq 2d-1$ implies that $|\cS_1|<d$.
	Since
	$$\mbox{rank}\bigg(\sum\limits_{j\in \cS_1}  \vu_j\vu_j^*\bigg)\leq |\cS_1|<d \ \ \mbox{and} \ \ 0\preceq 2\cdot\sum\limits_{j\in \cS_1}  \vu_j\vu_j^*\preceq 2\cdot\frac{n}{d}\mathbf{I},$$
there exists a unitary matrix $\mathbf{U}\in\mathbb{C}^{d\times d}$ such that
	$$
	\mathbf{U}^*\bigg(2\cdot\sum\limits_{j\in \cS_1}  \vu_j\vu_j^*\bigg) \mathbf{U}=\left(
	\begin{array}{ccccc}
		\lambda_1 & 0 & \cdots & 0 & 0 \\
		0 & \lambda_2 & \cdots & 0 & 0 \\
		\vdots& \vdots & \ddots & \vdots & \vdots \\
		0 & 0 & \cdots & \lambda_{d-1} & 0 \\
		0 & 0 & \cdots & 0 & 0 \\
	\end{array}
	\right),
	$$
	where $0\leq \lambda_i\leq 2\cdot\frac{n}{d},i=1,\ldots,d-1$. Hence
	$$
\frac{n}{d}	\cdot\mathbf{I}-2\cdot\sum\limits_{j\in \cS_1}  \vu_j\vu_j^*=\mathbf{U}\left(
	\begin{array}{ccccc}
		\frac{n}{d}-\lambda_1 & 0 & \cdots & 0 & 0 \\
		0 & \frac{n}{d}-\lambda_2 & \cdots & 0 & 0 \\
		\vdots& \vdots & \ddots & \vdots & \vdots \\
		0 & 0 & \cdots & \frac{n}{d}-\lambda_{d-1} & 0 \\
		0 & 0 & \cdots & 0 & \frac{n}{d} \\
	\end{array}
	\right)\mathbf{U}^*.
	$$
	Since $-\frac{n}{d}\leq \frac{n}{d}-\lambda_i\leq \frac{n}{d},i=1,\ldots,d-1$, we obtain that
	$$
	\bigg\|\sum\limits_{i\in \cS_2}  \vu_i\vu_i^*-\sum\limits_{j\in \cS_1}  \vu_j\vu_j^*\bigg\|=	\bigg\|\frac{n}{d}\cdot\mathbf{I}-2\cdot\sum\limits_{j\in \cS_1}  \vu_j\vu_j^*\bigg\|=\frac{n}{d}.
	$$
\end{proof}

Now we are ready to prove Theorem \ref{thm:sharp}.
\begin{proof}[Proof of Theorem \ref{thm:sharp}]
	Let $\{\vu_1,\ldots,\vu_n\}\subset\mathbb{C}^{d}$ be a tight frame with frame bound $C$.
	Since $\sum\limits_{i=1}^n\vu_i\vu_i^*=C\cdot\mathbf{I}$,  for any $\varepsilon_1,\ldots,\varepsilon_n\in\{-1,1\}$, we have
	$$
	- C \cdot \mathbf{I} \preceq \sum\limits_{i=1}^n\varepsilon_i\vu_i\vu_i^*\preceq C\cdot \mathbf{I}.
	$$
	Hence for any $\varepsilon_1,\ldots,\varepsilon_n\in\{-1,1\}$,
	$$
	\bigg\|\sum\limits_{i=1}^n\varepsilon_i\vu_i\vu_i^*\bigg\|\leq C,
	$$
	which implies
	$$
	\Disc(\mathbf{u}_1\mathbf{u}_1^*,\ldots,\mathbf{u}_n\mathbf{u}_n^*;\xi_1,\ldots,\xi_n)\leq C \leq \sqrt{\frac{n}{d}}\cdot \sqrt{\bigg\|\sum\limits_{i=1}^n(\vu_i\vu_i^*)^2\bigg\|}= \sqrt{\frac{n}{d}}\cdot\sigma.
	$$
We obtain (\ref{eq:tf}). Here, we use   Lemma \ref{lemma-2021-5-5}.

If $\{\vu_1,\ldots,\vu_n\}\subset\mathbb{C}^{d}$ forms a unit-norm tight frame and $d\leq n\leq 2d-1$, 
then we have
	$$
	\sigma^2=\bigg\|\sum\limits_{i=1}^n (\vu_i\vu_i^*)^2\bigg\|=\bigg\|\sum\limits_{i=1}^n \|\vu_i\|^2\vu_i\vu_i^*\bigg\|=\bigg\| \sum\limits_{i=1}^n \vu_i\vu_i^*\bigg\|=\frac{n}{d}.
	$$
 Lemma \ref{lemma-lower-321-1} implies  that for any  $\varepsilon_1,\ldots,\varepsilon_n\in\{-1,1\}$,
	$$
	\bigg\|\sum\limits_{i=1}^n \varepsilon_i \vu_i\vu_i^*\bigg\|=\frac{n}{d}=\sqrt{\frac{n}{d}}\cdot\sigma,
	$$
	which implies (\ref{eq:utf}).
\end{proof}

\section{Matrix discrepancy corresponding  to Schatten $p$-norm.}\label{sec:conc}
We end this paper by extending  matrix discrepancy to  a more generalized form, i.e.,  matrix discrepancy corresponding to Schatten $p$-norm. Recall that  the Schatten $p$-norm of a matrix $\mathbf{A}\in\mathbb{C}^{d\times d}$ is defined as
\[
\|\mathbf{A}\|_p = \bigg(\sum_{i=1}^d s_i(\mathbf{A})^p\bigg)^{1/p},
\]
where $s_i(\mathbf{A}), i = 1,\ldots,d$ are the singular values of $\mathbf{A}\in\mathbb{C}^{d\times d}$.
\begin{definition}\label{def:mdp}
	Let  $\mathbf{A}_1,\ldots,\mathbf{A}_n\in\mathbb{C}^{d\times d}$ and  $\xi_1,\ldots,\xi_n$ be independent random variables. Let $\cS_j$ be the support of $\xi_j$ for $1\leq j\leq n$. We define the matrix discrepancy of
	 matrices $\mathbf{A}_1,\ldots,\mathbf{A}_n$ and random variables $\xi_1,\ldots,\xi_n$  corresponding  to Schatten $p$-norm as
		$$
	\Disc_p(\mathbf{A}_1,\ldots,\mathbf{A}_n; \xi_1,\ldots,\xi_n)\,\,:=\,\,\min_{\varepsilon_1\in \cS_1,\ldots,\varepsilon_n\in \cS_n}\bigg\| \sum_{j=1}^n \varepsilon_j\mathbf{A}_j-\sum_{j=1}^n\E[\xi_j]\mathbf{A}_j \bigg\|_p.
	$$
\end{definition}
 	Note that the main results in this paper correspond the case of $p = \infty$ in Definition \ref{def:mdp}.
	For $ p\in[2, \infty)$, an upper bound  of matrix discrepancy corresponding to  Schatten $p$-norm can be derived by  the matrix Khintchine inequality. For more extensions and applications of the matrix Khintchine inequality, we refer the reader to \cite{MJCFT2014,Tro18,ivanov2020approximation}.
	\begin{lemma}[\cite{MJCFT2014}, Theorem 7.1]
		\label{lemma-MJC2014}
		Suppose that $2\leq p<\infty$. Consider a finite sequence $\{\mathbf{Y}_k\}_{k\geq1}$ of independent random Hermitian matrices  for which
		$
		\mathbb{E} \mathbf{Y}_k=\mathbf{0}.
		$
		Then
		$$
		\mathbb{E}\bigg\|\sum\limits_k \mathbf{Y}_k\bigg\|_{p}^{p}\leq \bigg(\frac{p-1}{2}\bigg)^{p/2}\cdot \mathbb{E}\bigg\|\bigg(\sum\limits_k\big(\mathbf{Y}_k^2+\mathbb{E}\mathbf{Y}_k^2)\bigg)^{1/2}\bigg\|_p^p.
		$$
	\end{lemma}

We can  employ Lemma \ref{lemma-MJC2014} to obtain the following result.

\begin{corollary}
\label{main-thm-schatten-norm}
Suppose that $2\leq p<\infty$ and $\xi_1,\ldots,\xi_n$ are independent scalar random variables with finite support.
Let $\mathbf{A}_1,\ldots,\mathbf{A}_n\in\mathbb{C}^{d\times d}$ be Hermitian matrices.
Then
$$
\Disc_p(\mathbf{A}_1,\ldots,\mathbf{A}_n; \xi_1,\ldots,\xi_n)\leq \sqrt{\frac{p-1}{2}}\cdot \bigg(\mathbb{E}\bigg\|\bigg(\sum\limits_{i=1}^n\big((\xi_i-\mathbb{E}[\xi_i])^2\mathbf{A}_i^2+({\bf Var}[\xi_i]\mathbf{A}_i)^2\big)\bigg)^{1/2}\bigg\|^p_p\bigg)^{1/p}.
$$
\end{corollary}
\begin{proof}
By using Lemma \ref{lemma-MJC2014} with $\mathbf{Y}_i=(\xi_i-\mathbb{E}[\xi_i])\mathbf{A}_i$ for $i\in[n]$, we have
$$
\mathop{\mathbb{E}}\limits_{\xi_1,\ldots,\xi_n}\bigg\|\sum\limits_{i=1}^n(\xi_i-\mathbb{E}[\xi_i])\mathbf{A}_i\bigg\|_p^p
			\leq \bigg(\frac{p-1}{2}\bigg)^{p/2}\cdot\mathop{\mathbb{E}}\limits_{\xi_1,\ldots,\xi_n}\bigg\|\bigg(\sum\limits_{i=1}^n\big((\xi_i-\mathbb{E}[\xi_i])^2\mathbf{A}_i^2+({\bf Var}[\xi_i]\mathbf{A}_i)^2\big)\bigg)^{1/2}\bigg\|^p_p.
$$
By the nonnegativity of the Schatten $p$-norm and the definition of the expectation, we know that there exists a choice of $\varepsilon_1,\ldots,\varepsilon_n$  in the support of $\xi_1,\ldots,\xi_n$ such that
		$$
		\bigg\|\sum\limits_{i=1}^n\big(\varepsilon_i-\mathbb{E}[\xi_i]\big)\mathbf{A}_i\bigg\|_p^p\leq \mathop{\mathbb{E}}\limits_{\xi_1,\ldots,\xi_n}\bigg\|\sum\limits_{i=1}^n(\xi_i-\mathbb{E}[\xi_i])\mathbf{A}_i\bigg\|_p^p.
		$$
		Then we have
		$$
		\Disc_p(\mathbf{A}_1,\ldots,\mathbf{A}_n; \xi_1,\ldots,\xi_n)\leq \sqrt{\frac{p-1}{2}}\cdot \bigg(	\mathop{\mathbb{E}}\bigg\|\bigg(\sum\limits_{i=1}^n\big((\xi_i-\mathbb{E}[\xi_i])^2\mathbf{A}_i^2+({\bf Var}[\xi_i]\mathbf{A}_i)^2\big)\bigg)^{1/2}\bigg\|^p_p\bigg)^{1/p}.
		$$
\end{proof}

Particularly, we have the following result for  independent Rademacher random variables.
\begin{corollary}\label{cor:kh1}
	Suppose that $2\leq p< \infty$ and $\xi_1,\ldots,\xi_n$ are independent  Rademacher random variables. Let $\mathbf{A}_1,\ldots,\mathbf{A}_n\in\mathbb{C}^{d\times d}$ be Hermitian matrices and
	$$
	\sigma= \bigg\|\bigg(\sum\limits_{i=1}^n \mathbf{A}_i^2 \bigg)^{\frac{1}{2}}\bigg\|_p.
	$$
	Then
	$$
	\Disc_p(\mathbf{A}_1,\ldots,\mathbf{A}_n; \xi_1,\ldots,\xi_n)\leq \sqrt{p-1}\cdot \sigma.
	$$
\end{corollary}
\begin{proof}
Since  for every $i=1,\ldots,n$, we have
 $$(\xi_i-\mathbb{E}[\xi_i])^2=1\ \ \mbox{and} \ \ {\bf Var}[\xi_i]=1,$$
 which derives
 $$
 \sum\limits_{i=1}^n\big((\xi_i-\mathbb{E}[\xi_i])^2\mathbf{A}_i^2+({\bf Var}[\xi_i]\mathbf{A}_i)^2\big)=2\cdot\sum\limits_{i=1}^n\mathbf{A}_i^2.
 $$
 Then by Corollary \ref{main-thm-schatten-norm}, we have this corollary.
\end{proof}

If $p=2$, i.e. the Frobenius norm,  we can get the following result.  We use $\|\cdot\|_F$ to denote the Frobenius norm to avoid confusion.
\begin{corollary}\label{cor:kh2}
	Suppose that $\xi_1,\ldots,\xi_n$ are independent scalar random variables with finite support.
	Let $\mathbf{A}_1,\ldots,\mathbf{A}_n\in\mathbb{C}^{d\times d}$ be Hermitian matrices and
	$$
	\sigma=\bigg\|\bigg(\sum\limits_{i=1}^n \big(\mbox{\bf Var}[\xi_i]\mathbf{A}_i\big)^2\bigg)^{\frac{1}{2}}\bigg\|_F.
	$$
	Then
	$$
	\Disc_2(\mathbf{A}_1,\ldots,\mathbf{A}_n; \xi_1,\ldots,\xi_n)\leq \sigma.
	$$
\end{corollary}
\begin{proof} By using Corollary \ref{main-thm-schatten-norm} and noting that
$$
\begin{array}{ll}
	&\mathop{\mathbb{E}}\limits_{\xi_1,\ldots,\xi_n}\bigg\|\bigg(\sum\limits_{i=1}^n\big((\xi_i-\mathbb{E}[\xi_i])^2\mathbf{A}_i^2+({\bf Var}[\xi_i]\mathbf{A}_i)^2\big)\bigg)^{\frac{1}{2}}\bigg\|^2_F
	\\
	&=\mathop{\mathbb{E}}\limits_{\xi_1,\ldots,\xi_n}\mbox{Tr}\bigg(\sum\limits_{i=1}^n\big((\xi_i-\mathbb{E}[\xi_i])^2\mathbf{A}_i^2+({\bf Var}[\xi_i]\mathbf{A}_i)^2\big)\bigg)\\
	&=\mbox{Tr}\mathop{\mathbb{E}}\limits_{\xi_1,\ldots,\xi_n}\bigg(\sum\limits_{i=1}^n\big((\xi_i-\mathbb{E}[\xi_i])^2\mathbf{A}_i^2+({\bf Var}[\xi_i]\mathbf{A}_i)^2\big)\bigg)\\
	&=2\cdot\mbox{Tr}\bigg(\sum\limits_{i=1}^n({\bf Var}[\xi_i]\mathbf{A}_i)^2\bigg)\\
	&=2\cdot\bigg\|\bigg(\sum\limits_{i=1}^n \big(\mbox{\bf Var}[\xi_i]\mathbf{A}_i\big)^2\bigg)^{\frac{1}{2}}\bigg\|_F^2,
\end{array}
$$
we have this corollary.
\end{proof}

\bibliographystyle{alpha}

\end{document}